\documentclass[11pt]{amsart}
\usepackage{amsopn}
\usepackage{amssymb}
\usepackage{graphicx}
\usepackage[all]{xy}
\usepackage{amscd}
\usepackage{color}
\usepackage[colorlinks]{hyperref}

 \newtheorem{thm}{Theorem}[section]
 \newtheorem{cor}[thm]{Corollary}
 \newtheorem{lem}[thm]{Lemma}
 \newtheorem{prop}[thm]{Proposition}
 
 \newtheorem{defn}[thm]{Definition}
 \newtheorem{exam}[thm]{Example}
 
 \theoremstyle{definition}
 \theoremstyle{remark}

 \numberwithin{equation}{section}

\begin{document}

\title[subspace- diskcyclic sequences of linear operators]
{subspace- diskcyclic sequences of linear operators  }

\author{ M. R. Azimi }
\address{M. R. Azimi }
\email{mhr.azimi@maragheh.ac.ir}

\address{department of mathematics, faculty of science, university of maragheh, 55181-83111 golshahr, maragheh, iran}

\thanks{}



\date{}

\dedicatory{}

\commby{}


\begin{abstract}
A sequence $\{T_n\}_{n=1}^{\infty}$  of bounded linear operators
between separable Banach spaces $X, Y$ is called diskcyclic if there
exists a vector $x\in X$ such that the disk-scaled orbit  $\{\alpha
T_n x: n\in \mathbb{N}, \alpha \in\mathbb{C}, | \alpha | \leq 1\}$
is dense in $Y$. In the first section of this paper we study some
conditions that imply the diskcyclicity of $\{T_n\}_{n=1}^{\infty}$.
 In particular, a sequence $\{T_n\}_{n=1}^{\infty}$  of bounded linear
 operators on separable infinite dimensional Hilbert space
 $\mathcal{H}$ is called subspace-diskcyclic with respect to the
 closed subspace $M\subseteq \mathcal{H},$ if there
exists a vector $x\in \mathcal{H}$ such that the disk-scaled orbit
$\{\alpha T_n x: n\in \mathbb{N}, \alpha \in\mathbb{C}, | \alpha |
\leq 1\}\cap M$ is dense in $M$. In the second section we survey
some conditions and subspace-diskcyclicity criterion (analogue the
results obtained by the some  mathematicians
 in \cite{MR2261697,
MR2720700, MR1111569}) which are sufficient for the sequence
$\{T_n\}_{n=1}^{\infty}$ to be subspace-diskcyclic.
\end{abstract}

\maketitle
\section{\textbf{Introduction and Preliminaries}}
Let $X$ and $Y$  be separable Banach spaces. The set of all bounded
linear operators from $X$ to $Y$ is denoted by $\mathcal{L}(X, Y)$.
A sequence of operators $\{T_n\}_{n=1}^{\infty}\subseteq
\mathcal{L}(X, Y)$ is called hypercyclic if there exists a vector
$x\in X$ such that the set $\{T_n x: n\in \mathbb{N}\}$ is dense in
$Y$. Such a vector is called hypercyclic vector for the sequence of
operators $\{T_n\}_{n=1}^{\infty}$. We say that an operator
$T:X\rightarrow X$ is hypercyclic if the sequence of its iterates
$\{T^n\}_{n=1}^{\infty}$ is hypercyclic. Over the last two decades
hypercyclic operators have been widely studied. A good survey of
hypercyclic operators is the recent book \cite{MR2533318}.
Furthermore some survey articles such as \cite{MR2281650},
\cite{MR1111569}, \cite{MR1148021} and  \cite{MR1685272} are
important references in this subject. The hypercyclicity of sequence
of linear operators and its relevant criteria have been studied in
\cite{MR2261697}, \cite{MR1980114} and \cite{MR1111569}. In
\cite{MR2720700}, B.F. Madore and R.A. Martinez-Avendano introduced
the concept of subspace-hypercyclicity for a bounded linear operator
$T$ defined on Hilbert space $\mathcal{H}$ and they proved a
Kitai-like criterion that implies subspace-hypercyclicity.
\\Indeed, all these motivated us to study diskcyclicity and
subspace-diskcyclicity of sequences of linear operators.\\ Let
$\mathcal{H}$ be a separable infinite dimensional Hilbert space over
the field of complex numbers $\mathbb{C}$ and let $M$ be a closed
subspace of $\mathcal{H}$. A sequence of  operators
$\{T_n\}_{n=1}^{\infty}\subseteq \mathcal{L}(\mathcal{H})$ is called
subspace-diskcyclic if there exists a vector $x\in X$ such that the
intersection of  disk scaled orbit of $\{T_n\}_{n=1}^{\infty}$ and
$M$, $$\{\alpha T_n x: n\in \mathbb{N}, \alpha \in\mathbb{C}, |
\alpha | \leq 1\}\cap M$$ is dense in $M$. Such a vector is called
subspace-diskcyclic vector for the sequence of operators
$\{T_n\}_{n=1}^{\infty}$ with respect to $M$. The set of all
subspace-diskcyclic vectors for $\{T_n\}_{n=1}^{\infty}$ is denoted
by $DC(\{T_n\}_{n=1}^{\infty}, M)$. In particular we say that an
operator $T:X\rightarrow X$ is subspace-diskcyclic for some $M
\subseteq \mathcal{H}$ if the sequence $\{T^n\}_{n=1}^{\infty}$ is
subspace-diskcyclic for $M$. Although the notion of
subspace-diskcyclicity can be defined between different separable
Banach spaces, nevertheless we prefer to deal with the Hilbert
space. Note that if the operator $T$ is hypercyclic then the
underlying Banach space $X$ should be separable. In \cite{MR2533318}
it is shown that an operator $T: X\rightarrow X$ is hypercyclic if
and only if it is topologically transitive i.e., for any pair $U, V$
of nonempty open subsets of $X$  there exists  $n\in \mathbb{N} $
such that $T_n(\alpha U)\cap V\neq \emptyset.$ In the first section
of this paper we define the notion of the topologically transitivity
for the sequences of operators $\{T_n\}_{n=1}^{\infty}$ and then we
show that this is a necessary and sufficient condition for
$\{T_n\}_{n=1}^{\infty}$ to be hypercyclic. Many criteria for
hypercyclicity of $\{T_n\}_{n=1}^{\infty}$ have been studied in
\cite{MR2261697}, \cite{MR2632793}, \cite{MR1980114} and
\cite{MR1111569}.
\\In section 3 we introduce the concept of subspace-disk topologically
transitivity for the $\{T_n\}_{n=1}^{\infty}$ and then it shall be
shown that $\{T_n\}_{n=1}^{\infty}$ is subspace-diskcyclic if and
only if it is subspace-disk topologically transitive. In addition
some necessary and sufficient conditions, criterion and other
properties concerning  the subspace-diskcyclicity of sequences of
linear operators $\{T_n\}_{n=1}^{\infty}$ are studied.



\section{\textbf{Diskcyclic sequences of linear operators }}
\begin{defn}
A sequence of operators $\{T_n\}_{n=1}^{\infty}\subseteq
\mathcal{L}(X, Y)$ is called disk topologically transitive if for
any pair $U, V$ of nonempty open subsets of $X$ and $Y$
respectively, there are $n\in \mathbb{N} \ \mbox{and} \  \alpha
\in\mathbb{C}, 0<| \alpha | \leq 1$ such that
$$T_n(\alpha U)\cap V\neq \emptyset.$$
\end{defn}
\begin{lem}
Let $\{T_n\}_{n=1}^{\infty}\subseteq \mathcal{L}(X, Y)$. Then
$$DC(\{T_n\}_{n=1}^{\infty})= \bigcap_{k=1}^{\infty}
\bigcup_{n=1}^{\infty} \bigcup_{\alpha \in \mathbb{C}, |\alpha|\geq
1} T_n^{-1}(\alpha V_k)$$ where $\{V_k\}$ is a countable open basis
for $Y$.
\end{lem}
\begin{proof}
Note that  $x\in DC(\{T_n\}_{n=1}^{\infty})$ if and only if for each
$k\in \mathbb{N}$, there exist $n\in \mathbb{N} \ \mbox{and}\ \alpha
\in\mathbb{C},\  | \alpha | \leq 1$ such that $\alpha T_n x\in V_k$
or $x\in T^{-1}_n(\alpha V_k)$. This occurs if and only if $x\in
\bigcap_{k=1}^{\infty} \bigcup_{n=1}^{\infty} \bigcup_{\alpha \in
\mathbb{C}, |\alpha|\geq 1} T_n^{-1}(\alpha V_k).$  Hence the set of
all diskcyclic vectors for $\{T_n\}_{n=1}^{\infty}$ is a $G_\delta$
set.
\end{proof}
\begin{lem}\label{l1}
A sequence of operators $\{T_n\}_{n=1}^{\infty}\subseteq
\mathcal{L}(X, Y)$ is diskcyclic if and only if it is disk
topologically transitive.
\end{lem}
\begin{proof}
Choose open subsets $U\subseteq X$ and $V\subseteq Y$ arbitrarily.
It is easy to check that if a sequence $\{T_n\}_{n=1}^{\infty})$ is
diskcyclic then $DC(\{T_n\}_{n=1}^{\infty})$ is dense in $X$. Hence
for any open subset $U$ of $X$ we have $$U \cap
DC(\{T_n\}_{n=1}^{\infty})\neq \emptyset .$$ Pick $x\in U \cap
DC(\{T_n\}_{n=1}^{\infty})\neq \emptyset.$ Then the set $\{\alpha
T_n x: n\in \mathbb{N}, \alpha \in\mathbb{C}, | \alpha | \leq 1\}$
is dense in $Y$ and so $\{\alpha T_n x: n\in \mathbb{N}, \alpha
\in\mathbb{C}, | \alpha | \leq 1\}\cap V\neq \emptyset .$  Thus,
$\alpha T_n x\in V$ for some $ n\in \mathbb{N}$ and $\alpha
\in\mathbb{C}, | \alpha | \leq 1$. Therefore $T_n(\alpha U)\cap
V\neq \emptyset .$\\
Conversely suppose that the sequence
$\{T_n\}_{n=1}^{\infty}\subseteq $ is disk topologically transitive.
 By Bair's category theorem and Lemma \ref{l1},
 $DC(\{T_n\}_{n=1}^{\infty})$ is dense $X$ if and only if every open
 set $W_k= \bigcup_{n=1}^{\infty} \bigcup_{\alpha \in \mathbb{C}, |\alpha|\geq
1} T_n^{-1}(\alpha V_k)$ is dense in $X$. Indeed, for every nonempty
open subset $U$ of $X$, there are $ n\in \mathbb{N}$ and $\alpha
\in\mathbb{C}, | \alpha | \geq 1$ such that $U\cap T^{-1}_n(\alpha
V_k)\neq \emptyset$, equivalently $T_n(\frac{1}{\alpha}U)\cap
V_k\neq \emptyset$ . This completes the proof.
\end{proof}
\begin{prop}
Let $\{T_n\}_{n=1}^{\infty}\subseteq \mathcal{L}(X, Y)$ be a
sequences of operators. The following conditions are equivalent:
\begin{itemize}
  \item [(i)] The sequence $\{T_n\}_{n=1}^{\infty}$ is disk
  topologically transitive;
  \item[(ii)] For each nonempty open subset $U$ of $X$ there are $\alpha \in \mathbb{C}, 0<|\alpha|
  \leq 1$ such that $\bigcup_{n=1}^{\infty}\bigcup_{|\alpha| \leq 1}
  T_n(\alpha
U)$ is dense in $Y$;
  \item [(iii)]For each nonempty open subset $V$ of $Y$ there are $\alpha \in \mathbb{C}, |\alpha|
  \geq 1$ such that $\bigcup_{n=1}^{\infty}\bigcup_{|\alpha| \geq 1}
  T_n^{-1}(\alpha V)$ is dense in $X$;
  \item [(iv)] For each $x\in X,$ $y\in Y$ and $\epsilon>0$ there
  exist $n\in
\mathbb{N} \ , \  \alpha \in\mathbb{C}, 0<| \alpha | \leq 1$ and
$u\in X$ such that $\|\alpha u-x\|<\epsilon$ and $\|\alpha T_{n}u
-y\|<\epsilon.$
\end{itemize}
\end{prop}
\begin{proof}
$(i)\Rightarrow (ii):$ Let $U$ be an arbitrary nonempty open subset
of $X$. By $(i)$ there exist $n\in \mathbb{N} \ \mbox{and} \  \alpha
\in\mathbb{C}, 0<| \alpha | \leq 1$ such that $T_n({\alpha}U)\cap
V\neq \emptyset$ for any nonempty open subset $V$ of $Y$. Hence
$(ii)$ is established.
\\$(ii)\Rightarrow(iii):$ Let $U\subseteq X$ and $V\subseteq Y$ be
any nonempty open subsets. By $(ii)$ there are $\alpha
\in\mathbb{C}, 0<| \alpha | \leq 1$ such that
$$\bigcup_{n=1}^{\infty}\bigcup_{|\alpha| \leq 1}
  T_n(\alpha U) \bigcap V\neq \emptyset.$$
Thus, there exist $n\in \mathbb{N} \ \mbox{and} \  \alpha
\in\mathbb{C}, 0<| \alpha | \leq 1$ such that $T_n({\alpha}U)\cap
V\neq \emptyset$. So $T^{-1}_n(\frac{1}{\alpha}V)\cap U\neq
\emptyset$ which implies that
$\bigcup_{n=1}^{\infty}\bigcup_{|\alpha| \geq 1}
  T_n^{-1}(\alpha V)$ is dense in $X$, since $U$ was chosen
  arbitrary.
  \\$(iii)\Rightarrow(i):$ We have  $\bigcup_{n=1}^{\infty}\bigcup_{|\alpha| \geq 1}
  T_n^{-1}(\alpha V)\bigcap U\neq \emptyset$  for every nonempty
  open subset $U$ of $X$. Therefore, $U\cap T^{-1}_n(\alpha
V)\neq \emptyset$ or  $T_n(\frac{1}{\alpha}U)\cap V\neq \emptyset$
for some $n\in \mathbb{N} \ \mbox{and} \  \alpha \in\mathbb{C}, |
\alpha | \geq 1$ and the  sequence $\{T_n\}_{n=1}^{\infty}$ is disk
  topologically transitive.
\\$(iv)\Rightarrow():$
\\$()\Rightarrow(iv):$
\end{proof}
\begin{lem}
Let $\{T_n\}_{n=1}^{\infty}\subseteq \mathcal{L}(X, Y)$ and $c_n\geq
0$ for $n=1, 2, ...$. If $\{c_n T_n\}_{n=1}^{\infty}$ is diskcyclic
then the sequence $\{k_n T_n\}_{n=1}^{\infty}$ is diskcyclic for all
$k_n\geq c_n.$
\end{lem}

\begin{proof}
Without loss of generality we may assume that $k_n>0$ for each $n\in
\mathbb{N}.$ Let $x$ be a diskcyclic vector for $\{c_n
T_n\}_{n=1}^{\infty}$. It is enough to show that$$\{\alpha c_nT_n x:
n\in \mathbb{N}, \alpha \in\mathbb{C}, | \alpha | \leq 1\}\subseteq
\{\alpha k_nT_n x: n\in \mathbb{N}, \alpha \in\mathbb{C}, | \alpha |
\leq 1\}.$$ Take $y\in \{\alpha c_nT_n x: n\in \mathbb{N}, \alpha
\in\mathbb{C}, | \alpha | \leq 1\}.$ Then there exist $n\in
\mathbb{N}$ and $\alpha \in\mathbb{C}, | \alpha | \leq 1$ such that
$y=\alpha c_nT_n x.$ One may write $y=\alpha \frac{c_n}{k_n}k_nT_n
x=\alpha'k_nT_n x$ where $\alpha \in\mathbb{C}, | \alpha' | \leq 1$.
This follows that $y\in \{\alpha k_nT_n x: n\in \mathbb{N}, \alpha
\in\mathbb{C}, | \alpha | \leq 1\}.$
\end{proof}
\section{\textbf{Subspace-diskcyclic sequences of linear operators }}
From now on $\mathcal{H}$ denotes a separable infinite dimensional
Hilbert space over the field of complex numbers $\mathbb{C}$.
\begin{defn}
Let $M$ be a nontrivial closed subspace of $\mathcal{H}$. A sequence
$\{T_n\}_{n=1}^{\infty}\subseteq \mathcal{L}(\mathcal{H})$ is called
\emph{subspace-diskcyclic} sequence of linear operators for $M$ if
there exists $x\in \mathcal{H}$ such that the set $$\{\alpha T_n x:
n\in \mathbb{N} \ , \  \alpha \in\mathbb{C}, 0<| \alpha | \leq
1\}\cap M$$ is dense in $M$. We call $x$ a
\emph{subspace-diskcyclic} vector for $\{T_n\}_{n=1}^{\infty}$. The
set of all subspace-diskcyclic vectors for $\{T_n\}_{n=1}^{\infty}$
in a subspace $M$ is denoted by $DC(\{T_n\}_{n=1}^{\infty}, M)$.
\end{defn}
\begin{exam}
One may consider that the subspace-diskcyclicity  does not imply
diskcyclicity in general. Let $\{T_n\}_{n=1}^{\infty}\subseteq
\mathcal{L}(\mathcal{H})$ be a diskcyclic sequence with the
diskcyclic vector $x$ and let $I$ be the indentity operator on
$\mathcal{H}$. Then the sequence
 $\{T_n\oplus I: \mathcal{H}\oplus \mathcal{H}\rightarrow \mathcal{H}\oplus \mathcal{H}\}_{n=1}^{\infty}$
  is subspace-diskcyclic for the subspace $M=\mathcal{H}\oplus \{0\}$ with the
  subspace-diskcyclic vector $x\oplus 0,$ while
  $\{T_n\oplus I\}_{n=1}^{\infty}$
   is not diskcyclic.
\end{exam}
\begin{thm}\label{T2}
Let $\{T_n\}_{n=1}^{\infty}\subseteq \mathcal{L}(\mathcal{H})$ and
let $M$ be a nontrivial subspace of $\mathcal{H}$. Then
$DC(\{T_n\}_{n=1}^{\infty}, M)= \bigcap_{k=1}^{\infty}
\bigcup_{|\alpha|\geq 1}\bigcup_{n=1}^{\infty} T_n^{-1}(\alpha B_k)$
where $\{B_k\}_{k=1}^{\infty}$ is a countable open basis for the
relatively topology of $M$ as a subspace of $\mathcal{H}$.
\end{thm}
\begin{proof}
Note that
\begin{eqnarray}
& & x\in DC(\{T_n\}_{n=1}^{\infty}, M)\nonumber \\
&\Leftrightarrow&  \{\alpha T_n x: n\in \mathbb{N} \ , \  \alpha
\in\mathbb{C}, 0<| \alpha | \leq
1\}\cap M is \  dense\  in\  M \nonumber \\
&\Leftrightarrow& for\ each \ k, \  there \ are \ n\in \mathbb{N} \
, \  \alpha \in \mathbb{C}, 0< | \alpha | \leq
1 \ such\ that\  \alpha T_n x \in B_k \nonumber \\
&\Leftrightarrow& for \ each\ k,\  x\in T^{-1}_n(\frac{1}{\alpha}B_k) \nonumber \\
&\Leftrightarrow& x\in \bigcap_{k=1}^{\infty} \bigcup_{|\alpha|\geq
1}\bigcup_{n=1}^{\infty} T_n^{-1}(\alpha B_k)\nonumber
\end{eqnarray}
\end{proof}
\begin{defn}
Let $\{T_n\}_{n=1}^{\infty}\subseteq \mathcal{L}(\mathcal{H})$ and
let $M$ be a nontrivial subspace of $\mathcal{H}$. We say that a
sequence of linear operators $\{T_n\}_{n=1}^{\infty}$ is
subspace-disk topologically transitive with respect to $M$ if for
all nonempty sets $U\subseteq M$ and $V\subseteq M,$ both relatively
open, there exist $n\in \mathbb{N}$ and $\alpha \in\mathbb{C}, |
\alpha | \geq 1$ such that $T^{-1}_n(\alpha U)\cap V$ contains a
relatively open nonempty subset of $M$.
\end{defn}
\begin{thm}\label{T1}
Let $\{T_n\}_{n=1}^{\infty}\subseteq \mathcal{L}(\mathcal{H})$ be a
sequence of linear operators and let $M$ be a nontrivial subspace of
$\mathcal{H}$. Then the following are equivalent:
\begin{itemize}
  \item [(i)] The sequence of linear operators $\{T_n\}_{n=1}^{\infty}$ is
subspace-disk topologically transitive with respect to $M$;
  \item [(ii)] for all nonempty sets $U\subseteq M$ and $V\subseteq M,$ both relatively
open, there exist $n\in \mathbb{N}$ and $\alpha \in\mathbb{C}, |
\alpha | \geq 1$ such that $T^{-1}_n(\alpha U)\cap V\neq \emptyset$
and $T_n M \subseteq M;$
  \item [(iii)] for all nonempty sets $U\subseteq M$ and $V\subseteq M,$ both relatively
open, there exist $n\in \mathbb{N}$ and $\alpha \in\mathbb{C}, |
\alpha | \geq 1$ such that $T^{-1}_n(\alpha U)\cap V$ is nonempty
open subset of $M$.
\end{itemize}
\end{thm}
\begin{proof}
$(i)\Rightarrow (ii):$ Let  $U\subseteq M$ and $V\subseteq M$
 be nonempty open subsets and let $W$ be the nonempty open subset of
$T^{-1}_n(\alpha U)\cap V$ for some $n\in \mathbb{N}$ and $\alpha
\in\mathbb{C}, | \alpha | \geq 1$. Then $\frac{1}{\alpha}T_n
W\subseteq M.$ Take $x\in M$ and $x_0\in W.$ Since $W$ is open
subset we may claim that $x_0 + rx\in W$ for sufficiently small
$r>0.$ Hence
$$\frac{1}{\alpha}T_n(x_0)+\frac{1}{\alpha}T_n(rx)=\frac{1}{\alpha}T_n(x_0 +rx)\in M.$$
Since $\frac{1}{\alpha}T_n(x_0)\in M$, it is easily inferred that
$T_n(x)\in M$ and the proof is complete.
 \\ The implication $(iii)\Rightarrow (i)$ is obvious.  $(ii)\Rightarrow
 (iii)$ is also obvious, since the sequence of operators
 $\{T_n|_M\}_{n=1}^{\infty}$ is still continuous.
\end{proof}
\begin{cor}
Let $\{T_n\}_{n=1}^{\infty}\subseteq \mathcal{L}(\mathcal{H})$ be a
sequence of linear operators and let $M$ be a nontrivial subspace of
$\mathcal{H}$. Assume that $\{T_n\}_{n=1}^{\infty}$ is subspace-disk
topologically transitive with respect to $M$. Then
$DC(\{T_n\}_{n=1}^{\infty}, M)$ is a dense subset of $M$.
\end{cor}

\begin{proof}
Let $\{B_i\}$ be a countable open basis for the relative topology of
$M$ as a subspace of $\mathcal{H}$. By Theorem \ref{T1}, for each
$i, j$, there exist $n_{i, j}\in \mathbb{N}$ and $\alpha_{i, j}
\in\mathbb{C}, | \alpha_{i, j} | \geq 1$ such that the set
$T^{-1}_{n_{i, j}}(\alpha_{i, j}B_i)\cap B_j$ is a nonempty open
subset of $M$. Hence the set $$A_i= \bigcup_j T^{-1}_{n_{i,
j}}(\alpha_{i, j}B_i)\cap B_j$$ is a nonempty, open and dense set in
$M$. By Bair's category theorem
$$\bigcap_i A_i= \bigcap_i\bigcup_j T^{-1}_{n_{i,
j}}(\alpha_{i, j}B_i)\cap B_j$$ remains still dense set in $M$. But
by Theorem \ref{T2}, we know that $$DC(\{T_n\}_{n=1}^{\infty}, M)=
\bigcap_{i} \bigcup_{n} \bigcup_{|\alpha|\geq 1}T^{-1}(\alpha B_i)$$
and the result is obtained.
\end{proof}
\begin{cor}
If $\{T_n\}_{n=1}^{\infty}$ is subspace-disk topologically
transitive for a subspace $M,$ then $\{T_n\}_{n=1}^{\infty}$ is
diskcyclic for $M$.
\end{cor}

\begin{proof}
\end{proof}
\begin{thm}
Let $\{T_n\}_{n=1}^{\infty}\subseteq \mathcal{L}(\mathcal{H})$ be a
sequence of linear operators and let $M$ be a nontrivial subspace of
$\mathcal{H}$. Assume that there exist $X$ and $Y$, dense subsets of
$M$ and an increasing sequence of positive integers
$\{n_k\}_{k=1}^{\infty}$ such that
\begin{itemize}
  \item [(i)] $T_{n_k}x\rightarrow 0$ for all $x\in X;$
  \item [(ii)] for any $y\in Y$,there exists a sequence $\{x_k\}$ in
  $M$ such that $x_k\rightarrow 0$ and $T_{n_k}x_k\rightarrow y;$
  \item [(iii)] $T_{n_k} M \subseteq M$ for each $k\in \mathbb{N}.$
\end{itemize}
Then $\{T_n\}_{n=1}^{\infty}$ is subspace-topologically transitive
with respect to $M$ and hence $\{T_n\}_{n=1}^{\infty}$ is
subspace-hypercyclic for $M.$
\end{thm}

\begin{proof}
The sketch  of the proof is well-known and we follow it same as used
in \cite{MR2720700}. Let  $U\subseteq M$ and $V\subseteq M$ be
nonempty open subsets. By Theorem \ref{T2} we have to only show that
there exists $k\in \mathbb{N}$  such that $T^{-1}_{n_k}(U)\cap V$ is
nonempty. Since $X$ and $Y$ are dense in $M$, there exists $u\in
U\cap Y$ and $v\in V\cap X.$ Moreover, one may catch $\delta >0$
such that the $M$-ball centered at $u$ of radius $\delta$, denoted
by $B_M(u, \delta)$, is contained in $U$ and $B_M(v,
\delta)\subseteq V.$ Now by $(ii)$, we can choose $k$ large enough
such that there exists $x_k\in M$ with $$\|T_{n_k}v\|<
\frac{\delta}{2}, \ \ \|x_k\|< \delta \ \ \mbox{and} \ \|T_{n_k}x_k-
u\|<\frac{\delta}{2}.$$ We know that $v+ x_k\in M$ and $$\|v+ x_k
-v\|=\|x_k\|< \delta$$ which follows that $$v+ x_k\in B_M(v,
\delta)\subseteq V.$$ \\In addition, $T_{n_k}$ leaves $M$ invariant,
so $T_{n_k}(v+x_k)\in M$ and
$$\|T_{n_k}(v+x_k)-u\|\leq \|T_{n_k} v\|+\|T_{n_k}x_k - u\|<\frac{\delta}{2}+\frac{\delta}{2}=\delta.$$
It follows that $$T_{n_k}(v+x_k)\in B_M(u, \delta) \subseteq U.$$
Eventually, the above arguments imply that $$v+ x_k\in
T^{-1}_{n_k}(U)\cap V$$ and the result follows.
\end{proof}
\begin{thm}
Let $\{T_n\}_{n=1}^{\infty}\subseteq \mathcal{L}(\mathcal{H})$ be a
sequence of linear operators and let $M$ be a nontrivial subspace of
$\mathcal{H}$. Assume that there exist $X$ and $Y$, dense subsets of
$M$ and an increasing sequence of positive integers
$\{n_k\}_{k=1}^{\infty}$ such that
\begin{itemize}
  \item [(i)] $\alpha T_{n_k}x\rightarrow 0$ for all $x\in X;$
  \item [(ii)] for any $y\in Y$,there exist a sequence $\{x_k\}$ in
  $M$ and $\alpha \in\mathbb{C}, |
\alpha | \leq 1$ such that $x_k\rightarrow 0$ and $\alpha
T_{n_k}x_k\rightarrow y;$
  \item [(iii)] $T_{n_k} M \subseteq M$ for each $k\in \mathbb{N}.$
\end{itemize}
Then $\{T_n\}_{n=1}^{\infty}$ is subspace-disk topologically
transitive with respect to $M$ and hence $\{T_n\}_{n=1}^{\infty}$ is
subspace-diskcyclic for $M.$
\end{thm}

\begin{proof}
Let $U$ and $V$ be nonempty relatively open subsets of $M$. By
Theorem \ref{T1}, it is enough to prove that there exist $k\in
\mathbb{N}$ and $\alpha \in\mathbb{C}, | \alpha | \geq 1$ such that
$T^{-1}_{n_k}(\alpha U)\cap V$ is nonempty.  For each $\epsilon
>0$, choose $k$ large enough such that there
exist $x_k\in M$ and $\alpha \in\mathbb{C}, 0<| \alpha | \leq 1$
where
$$\|T_{n_k}x\|< \frac{\epsilon}{2}, \ \ \|x_k\|< \epsilon \ \
\mbox{and} \ \|\alpha T_{n_k}x_k- y\|<\frac{\epsilon}{2}$$ hold for
every $x\in X$ and $y\in Y.$ As mentioned in the proof of the
previous theorem,  $\alpha u\in U\cap Y $ , $ v\in V\cap X$ and
$\delta>0$ are easily found on which
$$B_M(\alpha u, \delta)\subseteq U\ \mbox{and} \ B_M( v,
\delta)\subseteq V.$$ Hence the above inequalities can be rewritten
as follows $$\|T_{n_k}v\|< \frac{\delta}{2|\alpha|}, \ \ \|x_k\|<
\delta \ \ \mbox{and} \ \|\alpha T_{n_k}x_k- \alpha
u\|<\frac{\delta}{2}.$$ But $v+ x_k\in B_M(v, \delta)\subseteq V$
and $T_{n_k}(v+x_k)\in M$, since  $T_{n_k}$ leaves $M$ invariant.
Moreover
$$\|\alpha T_{n_k}(v+x_k)-\alpha u\|\leq \|\alpha T_{n_k} v\|+\|\alpha T_{n_k}x_k - \alpha u\|<
\frac{\delta}{2}+\frac{\delta}{2}=\delta$$ which  follows that
$$\alpha T_{n_k}(v+x_k)\in B_M(\alpha u, \delta) \subseteq U.$$
Therefore $T^{-1}_{n_k}(\frac{1}{\alpha} U)\cap V\neq \emptyset$ and
the proof is complete.
\end{proof}
\begin{thm}
Let $\{T_n\}_{n=1}^{\infty}\subseteq \mathcal{L}(\mathcal{H})$ be a
subspace-diskcyclic sequence of linear operators for a  nontrivial
subspace $M$ of $\mathcal{H}$. Suppose that $...\supseteq
M_3\supseteq M_2 \supseteq M_1\supseteq M$ is an invariant subspace
sequence for $\{T_n\}_{n=1}^{\infty}$ i.e., $T_n  M_n\subseteq M_n$.
Then $\{T_n|_{M_n}\}_{n=1}^{\infty}$ is a subspace-diskcyclic for
$M$.
\end{thm}

\begin{proof}
\end{proof}
\begin{thm}
Let $\mathcal{H}=M\oplus N$ and $P$ be the projection onto $M$ along
$N$. Let $T_n N\subseteq N$ for each $n\in \mathbb{N}$. If
$\{T_n\}_{n=1}^{\infty}\subseteq \mathcal{L}(\mathcal{H})$ is
subspace-diskcyclic for some $L\subseteq M$, then
$\{PT_n|_{M}\}_{n=1}^{\infty}$ is subspace-diskcyclic for $L$.
\end{thm}

\begin{proof}
Suppose that $\{T_n\}_{n=1}^{\infty}$ is subspace-diskcyclic for
$L\subseteq M$ with diskcyclic vector $x\in L.$ Then
$$\{\alpha T_n x:
n\in \mathbb{N} \ , \  \alpha \in\mathbb{C}, 0<| \alpha | \leq
1\}\cap L $$$$\subseteq P(\{\alpha T_n x: n\in \mathbb{N} \ , \
\alpha \in\mathbb{C}, 0<| \alpha | \leq 1\})\cap L.$$ We have $PT_n
P=PT_n,$ since $N$ is invariant for every $T_n$. This implies that
$$P(\{\alpha T_n x: n\in \mathbb{N} \ , \  \alpha
\in\mathbb{C}, 0<| \alpha | \leq 1\})$$$$=\{\alpha PT_n|_M x: n\in
\mathbb{N} \ , \  \alpha \in\mathbb{C}, 0<| \alpha | \leq 1\}.$$
Therefore $\{\alpha PT_n|_M x: n\in \mathbb{N} \ , \  \alpha
\in\mathbb{C}, 0<| \alpha | \leq 1\}$ is dense in $L$.
\end{proof}
\begin{cor}
Let $M$ be a reducible subspace for any $T_n$. If
$\{T_n\}_{n=1}^{\infty}\subseteq \mathcal{L}(\mathcal{H})$ is
subspace-diskcyclic for some $L\subseteq M$, then
$\{PT_n\}_{n=1}^{\infty}$ is subspace-diskcyclic for $L$.
\end{cor}

\begin{proof}
\end{proof}
\begin{thm}
Let $\mathcal{H}=\bigoplus_n M_n$ and $P_n$ be the projection onto
$M_n$ along $\bigvee_{i\neq n}M_i$. For each $n\in \mathbb{N}$
assume that $T_nM_i\subseteq M_i$ $(i=1, 2, 3, ...)$. If a sequence
$\{T_n\}_{n=1}^{\infty}\subseteq \mathcal{L}(\mathcal{H})$ is
subspace-diskcyclic for some $M_n$, then $\{P_nT_n\}_{n=1}^{\infty}$
is subspace-diskcyclic for $M_n$.
\end{thm}

\begin{proof}
\end{proof}

\nocite{*}
\bibliographystyle{amsplain}

\end{document}